\documentclass[12pt]{article} 
\usepackage{amsfonts,amsmath,latexsym,amssymb,mathrsfs,amsthm,comment,setspace}
\usepackage{slashbox}
\usepackage{caption}

\let\OLDthebibliography\thebibliography
\renewcommand\thebibliography[1]{
  \OLDthebibliography{#1}
  \setlength{\parskip}{0pt}
  \setlength{\itemsep}{0pt plus 0.0ex}
}

\def\numberlikeadb{\global\def\theequation{\thesection.\arabic{equation}}}
\numberlikeadb
\newtheorem{theorem}{Theorem}[section]

\newtheorem{corollary}[theorem]{Corollary}

\newtheorem{proposition}[theorem]{Proposition}
\newtheorem{remark}[theorem]{Remark}

\usepackage{color}

\usepackage{lscape}
\usepackage{caption}
\usepackage{multirow}
\begin{document}

\title{Stein's method and the distribution of the product of zero mean correlated normal random variables}
\author{Robert E. Gaunt\footnote{School of Mathematics, The University of Manchester, Manchester M13 9PL, UK, robert.gaunt@manchester.ac.uk}}

\date{} 
\maketitle

\vspace{-10mm}


\begin{abstract}Over the last 80 years there has been much interest in the problem of finding an explicit formula for the probability density function of two zero mean correlated normal random variables.  Motivated by this historical interest, we use a recent technique from the Stein's method literature to obtain a simple new proof, which also serves as an exposition of a general method that may be useful in related problems.

\end{abstract}




\noindent{{{\bf{AMS 2010 Subject Classification:}}} Primary 60E05; 62E15

\section{Introduction}

Let $(X,Y)$ be a bivariate normal random vector with zero mean vector, variances $(\sigma_X^2,\sigma_Y^2)$ and correlation coefficient $\rho$.  The exact distribution of the product $Z=XY$ has been studied since 1936 (Craig 1936); with contributions including the works of Aroian \cite{aa1}, Aroian, Taneja and Cornwell \cite{aa2}, Bandi and Connaughton \cite{bandi}, Haldane \cite{hh1}, Meeker et al$.$ \cite{mm1}; see Nadarajah and Pog\'{a}ny \cite{np16} for an overview of these and further contributions.  The  distribution of $Z$ has been used in numerous applications since 1936, with some recent examples being: product confidence limits for indirect effects (MacKinnon et al$.$ \cite{mac2}); statistics of Lagrangian power in two-dimensional turbulence (Bandi and Connaughton \cite{bandi}); statistical mediation analysis (MacKinnon \cite{mac1}}{) and electrical engineering (Ware and Lad \cite{ware}).  Ware and Lad \cite{ware} have also provided three different methods of numerical integration for computing the probability that $Z$, and more generally sums of independent variates with the same distribution as $Z$, take a negative value.  However, despite this interest, the problem of finding an exact formula for the probability density function (PDF) of $Z$ remained open for many years.

Recently in 2016, some 80 years after the problem was first studied, Nadarajah and Pog\'{a}ny \cite{np16} used an approach based on characteristic functions to obtain an explicit formula for the PDF of $Z$.  As a by-product, an explicit formula was obtained for the PDF of the mean $\overline{Z}=\frac{1}{n}(Z_1+\cdots+Z_n)$, where $Z_1,\ldots, Z_n$ are independent and identical copies of $Z$: for $n\geq1$ variates,
\begin{equation}\label{pdf12}p(x)=\frac{n^{(n+1)/2}2^{(1-n)/2}|x|^{(n-1)/2}}{(\sigma_X\sigma_Y)^{(n+1)/2}\sqrt{\pi(1-\rho^2)}\Gamma\big(\frac{n}{2}\big)}\exp\bigg(\frac{\rho n x}{\sigma_X\sigma_Y(1-\rho^2)} \bigg)K_{\frac{n-1}{2}}\bigg(\frac{n |x|}{\sigma_X\sigma_Y(1-\rho^2)}\bigg), 
\end{equation}
$x\in\mathbb{R}$, where $K_\nu(x)=\int_0^\infty \mathrm{e}^{-x\cosh(t)}\cosh(\nu t)\,\mathrm{d}t$ is a modified Bessel function of the second kind. Of course, on setting $n=1$ we recover the formula for the product $Z$. 

Since the work of Nadarajah and Pog\'{a}ny \cite{np16}, the distributions of $Z$ and $\overline{Z}$ were identified as variance-gamma random variables by Gaunt \cite{gaunt prod}, from which a formula for the PDFs was immediate.  Also, an exact formula for the PDF of a product of  correlated normal random variables with non-zero means was recently obtained by Cui et al$.$ \cite{cui}.  This formula takes a complicated form, involving a double sum of modified Bessel functions of the second kind.

In this note, we use a recent technique from the Stein's method literature to obtain a new derivation of the PDF of $\overline{Z}$.  This was in part motivated by the historical interest in this problem, but it also serves as a useful exposition of a neat technique for finding PDFs.  The proof of Gaunt \cite{gaunt prod} is very simple but relies on identifying $\overline{Z}$ as a variance-gamma random variable and then exploiting results from the distributional theory of such random variables.  The advantage of the approach given in this paper over that of  Gaunt \cite{gaunt prod} is that we do not need to appeal to such a theory, and are able to give a simple self-contained proof.
Also, our proof gives a transparent explanation as to why the modified Bessel function $K_\nu(x)$ occurs in the density in that our approach involves finding a second order differential equation satisfied by the PDF that closely resembles the modified Bessel differential equation.  This is not so clear from the proof of Gaunt \cite{gaunt prod}, nor that of Nadarajah and Pog\'{a}ny \cite{np16} which involves an application of the Fourier inversion formula and the evaluation of certain complex integrals in terms of $K_\nu(x)$.  


Introduced in 1972, Stein's method (Stein \cite{stein}) is a powerful technique for deriving distributional approximations in probability theory.  At the heart of the method is a \emph{Stein characterisation} of the target distribution.  For the normal distribution: $X\sim N(\mu,\sigma^2)$ if and only if 
\begin{equation}\label{norstein}\mathbb{E}[\sigma^2g'(X)-(X-\mu)g(X)]=0
\end{equation}
for all differentiable $g:\mathbb{R}\rightarrow\mathbb{R}$ such that $\mathbb{E}|g'(Y)|$, $\mathbb{E}|Yg(Y)|$ and $\mathbb{E}|g(Y)|$ are all finite for $Y\sim N(\mu,\sigma^2)$.  We note that necessity is obtained almost immediately from the differential equation $\sigma^2 \phi(x)'+(x-\mu)\phi(x)=0$ that the $N(\mu,\sigma^2)$ PDF $\phi(x)$ solves: just multiply through by $g(x)$, integrate over $\mathbb{R}$ and then integrate by parts.  Over the years, Stein characterisations have been obtained for many classical probability distributions (for an overview see Gaunt, Mijoule and Swan \cite{gms16} and Ley, Reinert and Swan \cite{ley}), and also recently for more exotic distributions, such as linear combinations of gamma random variables (Arras et al$.$ \cite{aaps18}) and products of independent normal, beta and gamma random variables (Gaunt \cite{gaunt pn, gaunt18}), for which it is difficult to write down a formula for the PDF of the distribution.  

Stein characterisations of probability distributions are most commonly used as part of Stein's method to derive distributional approximations, with powerful applications in random graph and network theory (Franceschetti and Meester \cite{steinnet}), convergence rates in classical asymptotic results in statistics (Anastasiou and Reinert \cite{ag17}, Gaunt, Pickett and Reinert \cite{gaunt chi square}), Bayesian statistics (Ley, Reinert and Swan \cite{ley2}) and statistical learning and inference (Gorham et al$.$ \cite{samplequal}); see the survey Ross \cite{ross} for a list of further application areas.  However, recently Gaunt \cite{gaunt18} and Gaunt, Mijoule and Swan \cite{gms16} have found a novel application for Stein characterisations, in which they are used to establish formulas for PDFs of distributions that are too difficult to obtain via other methods.  The basic approach, which we shall employ in this note, is to obtain a Stein characterisation of $\overline{Z}$ and then apply integration by parts to the characterising equation to deduce an ordinary differential equation (ODE) that the PDF must satisfy, from which we easily obtain the formula (\ref{pdf12}) for the density.  



\section{Proof of (\ref{pdf12}) via a Stein characterisation of the distribution}\label{sec2}

Here, we provide an alternative proof of the main results of Nadarajah and Pog\'{a}ny \cite{np16}.  Throughout, we shall set $\sigma_X^2=\sigma_Y^2=1$; the extension to the general case is straightforward.  The starting point is the following Stein characterisation of $\overline{Z}$.  Often in the Stein's method literature a full characterisation of distributions is given, as given for the normal distribution in (\ref{norstein}).  However, for our purposes we only require necessity.

\begin{proposition}\label{prop1}Suppose $f:\mathbb{R}\rightarrow\mathbb{R}$ is twice differentiable with $\mathbb{E}|f(\overline{Z})|$, $\mathbb{E}|\overline{Z}f(\overline{Z})|$, $\mathbb{E}|f'(\overline{Z})|$, $\mathbb{E}|\overline{Z}f'(\overline{Z})|$ and $\mathbb{E}|\overline{Z}f''(\overline{Z})|$ all finite. Then
\begin{equation}\label{char1}\mathbb{E}[(1-\rho^2)\overline{Z}f''(\overline{Z})+\tfrac{1}{n}((1-\rho^2)+2\rho \overline{Z})f'(\overline{Z})+(\rho-\overline{Z})f(\overline{Z})]=0.
\end{equation}
\end{proposition}



\begin{proof}We first establish the result for $n=1$ before extending to $n\geq1$.  Define the random variable $V$ by $V=(Y-\rho X)/\sqrt{1-\rho^2}$, which is readily seen to be standard normally distributed and independent of $X$. Then, we can write $Z=\sqrt{1-\rho^2}VX+\rho X^2$.  Therefore, $Z\,|\,X\sim N(\rho X^2,(1-\rho^2)X^2)$, and we obtain from (\ref{norstein}) that
\begin{align}\mathbb{E}[Zf(Z)]&=\mathbb{E}[\mathbb{E}[Zf(Z)\,|\,X]]\nonumber\\
&=\mathbb{E}[\mathbb{E}[(1-\rho^2)X^2f'(Z)+\rho X^2f(Z)\,|\,X]]\nonumber \\ 
\label{sub5}&=\mathbb{E}[(1-\rho^2)X^2f'(Z)+\rho X^2f(Z)].
\end{align}
 Let $z=\sqrt{1-\rho^2}vx+\rho x^2$, and note that $\frac{\partial }{\partial x}\big(xf(z)\big)=(z+\rho x^2)f'(z)+f(z)$.
Then, on using the Stein characterisation of the normal distribution (\ref{norstein}) with $\mu=0$, $\sigma^2=1$ and $g(x)=(1-\rho^2)xf'(z)+\rho xf(z)$ to obtain the second equality, we have that
\begin{align*}\mathbb{E}[Zf(Z)]&=\mathbb{E}[\mathbb{E}[(1-\rho^2)X\cdot Xf'(Z)+\rho X\cdot X f(Z)\,|\,V]] \\
&=\mathbb{E}[\mathbb{E}[(1-\rho^2)Zf''(Z)+(1-\rho^2)f'(Z)+\rho Zf'(Z)+\rho f(Z)\\
&\quad+(1-\rho^2)\rho X^2 f''(Z)+\rho^2X^2f'(Z)\,|\,V]] \\
&=\mathbb{E}[(1-\rho^2)Zf''(Z)+(1-\rho^2)f'(Z)+\rho Zf'(Z)+\rho f(Z)\\
&\quad+(1-\rho^2)\rho X^2 f''(Z)+\rho^2X^2f'(Z)] \\
&=\mathbb{E}[(1-\rho^2)Zf''(Z)+(1-\rho^2+2\rho Z)f'(Z)+\rho f(Z)],
\end{align*}
where the final equality follows from (\ref{sub5}) with $f$ replaced by $f'$. 

Now, we extend to $n\geq1$.  Let $W=n\overline{Z}=\sum_{i=1}^n Z_i$.  Then, by conditioning,
\begin{align}\mathbb{E}[(W-n\rho)f(W)]&=\sum_{i=1}^n\mathbb{E}[\mathbb{E}[(Z_i-\rho)f(W)\,|\,\{Z_j\}_{j\not=i}]] \nonumber\\
&=\sum_{i=1}^n\mathbb{E}[\mathbb{E}[(1-\rho^2)Z_if''(W)+(1-\rho^2+2\rho Z_i)f'(W)\,|\,\{Z_j\}_{j\not=i}]] \nonumber\\
\label{charsub}&=\mathbb{E}[(1-\rho^2)Wf''(W)+(n(1-\rho^2)+2\rho W)f'(W)],
\end{align}
and on substituting $f(x)=g(x/n)$ into (\ref{charsub}) we obtain (\ref{char1}).
\end{proof}

\begin{corollary}The PDF $p(x)$ of $\overline{Z}$ satisfies the ODE
\begin{equation}\label{ode1}(1-\rho^2)xp''(x)-\tfrac{1}{n}((1-\rho^2)+2\rho x)p'(x)+(\rho-x)p(x)=0.
\end{equation}
\end{corollary}

\begin{proof}Let $f$ be defined as in Proposition \ref{prop1}, and denote this class of functions $\mathcal{F}$.  Then, applying integration by parts to (\ref{char1}) gives that
\begin{align}\label{int0}\int_{-\infty}^\infty \big\{(1-\rho^2)xp''(x)-\tfrac{1}{n}((1-\rho^2)+2\rho x)p'(x)+(\rho-x)p(x)\big\}f(x)\,\mathrm{d}x=0.
\end{align}
Since (\ref{int0}) holds for all $f\in\mathcal{F}$, we deduce that $p$ satisfies the ODE (\ref{ode1}).
\end{proof}

\noindent{\emph{Proof of (\ref{pdf12}).}} The general solution to (\ref{ode1}) is given by
\begin{align*}p(x)&=A|x|^{(n-1)/2}\mathrm{e}^{\rho n x/(1-\rho^2)}K_{\frac{n-1}{2}}\bigg(\frac{n |x|}{1-\rho^2}\bigg) +B|x|^{(n-1)/2}\mathrm{e}^{\rho n x/(1-\rho^2)}I_{\frac{n-1}{2}}\bigg(\frac{n |x|}{1-\rho^2}\bigg), 
\end{align*}
where $A$ and $B$ are arbitrary constants and $I_\nu(x)=\sum_{k=0}^\infty\frac{(x/2)^{2k+\nu}}{k!\Gamma(k+\nu+1)}$ is a modified Bessel function of the first kind.  That this is the general solution can be deduced from the fact that the general solution to the modified Bessel differential equation
\[x^2h''(x)+xh'(x)-(x^2+\nu^2)h(x)=0\]
is given by $h(x)=CK_\nu(x)+DI_\nu(x)$ (see Olver et al$.$ 2010).  Now, for $p$ to be a PDF we require that $\int_{-\infty}^\infty p(x)\,\mathrm{d}x=1$.  But, for any $\nu\in\mathbb{R}$, as $x\rightarrow\infty$, $I_\nu(x)\sim\frac{1}{\sqrt{2\pi x}}\mathrm{e}^x$, and so we must take $B=0$.  We can also find $A$ by using the integral formula (which can be obtained by using the series expansion $\mathrm{e}^{\beta x}=\sum_{k=0}^\infty\frac{1}{k!}(\beta x)^k$ followed by formula (10.43.19) of Olver et al$.$ \cite{olver})
\begin{equation} \label{pdfk} \int_{-\infty}^{\infty}\mathrm{e}^{\beta x} |x|^{\nu} K_{\nu}(|x|)\,\mathrm{d}x =\frac{\sqrt{\pi}\Gamma(\nu+1/2)2^{\nu}}{(1-\beta^2)^{\nu+1/2}}, \quad \nu>-\tfrac{1}{2}, \: -1<\beta <1,
\end{equation}
from which we deduce that the PDF of $\overline{Z}$ is given by (\ref{pdf12}). \hfill $\Box$

\begin{remark}An interesting way in which this work could be extended would be to repeat the analysis in the more general setting of the product of two correlated normal variates with non-zero mean vector; the exact distribution of this random variable is not known in the current literature. This is in principle possible, but one faces additional technical difficulties.  A Stein characterisation for this distribution in the special case of zero correlation (for which the PDF has already be given by  Cui et al$.$ \cite{cui}) is given in Proposition 3.3 of Gaunt, Mijoule and Swan \cite{gms19}.  Applying the same argument that has been used in this paper yields a fourth order ODE that the density must satisfy; see Section 3.3.2 of Gaunt, Mijoule and Swan \cite{gms19}.  As noted, in that work, it is a difficult task to solve the ODE, and this would be even more challenging for the ODE corresponding to the more general uncorrelated case.
\end{remark}


\subsection*{Acknowledgements}
The author is supported by a Dame Kathleen Ollerenshaw Research Fellowship.  The author would like to thank the referees for their helpful comments.

\end{document}